\documentclass[10pt]{article}
\usepackage{amsmath, mathtools, amssymb}
\usepackage{amsthm}
\usepackage{graphicx}
\usepackage{setspace}
\usepackage{hyperref}
\def\ds{\displaystyle}
\def\ni{\noindent}
\newcommand{\sinn}{\sum\limits_{n=-\infty}^{\infty}}
\newcommand{\szn}{\sum\limits_{n=0}^{\infty}}

\newcommand{\non}{\sum\limits_{n=1}^{\infty}}

\newcommand{\et}[1]{\eta({{#1}\tau)}}

\newcommand{\legendre}[2]{\genfrac{(}{)}{}{}{#1}{#2}}
\newtheorem{thm}{Theorem}

\numberwithin{thm}{section}
\numberwithin{cor}{section}
\numberwithin{dfn}{section}
\numberwithin{lemma}{section}

\title{\textbf{***}}
\date{}

\numberwithin{equation}{section}
\begin{document}
	\begin{center}
		\Large{\textbf{On Oliver's relations between infinite series and infinite products}}
	\end{center}
	
	\begin{center}
		{\textbf{K. R. Vasuki$^{1}$, Darshan D.$^{2}$, and Ravi G. N.$^{3}$ }} \vskip 0.1 cm \ni \small \vskip 0.1 cm \ni \small $^1 $$^{,}$$ ^{2} $$^{,}$Department of Studies in Mathematics, University of Mysore,\\
		Manasagangotri, Mysuru - 570 006, INDIA.\\
			$~~~~^1$E-mail: \textit{vasuki$\_$kr@hotmail.com }\vskip 0.1 cm \ni \small  $~~~~~~~~~^{2}$E-mail:  \textit{darshand13057@gmail.com} \\
		$ ^{3}$Department of Mathematics, Government College for Women,\\ Kolar - 563101, INDIA.\\
		\vskip 0.1 cm \ni \small  $~~~^{3}$E-mail:  \textit{gnravi1993@gmail.com}
		
	\end{center}
	
	\begin{abstract}
	The Jacobi triple product identity provides closed-form expressions for many infinite-product generating functions that arise naturally in combinatorics and number theory. A particularly important application is to Dedekind's eta function $\eta(\tau)$, which it expresses as a theta function. Using the theory of modular forms, Oliver \cite{oliver} classified all eta quotients that are theta functions. In this paper, we present elementary proofs of the identities established by Oliver.

	\end{abstract}
	
	\noindent \textbf{Keywords:} Theta functions, eta functions, Legendre symbol.  \\
	
	\noindent \textbf{MSC Classification 2020:} 11F20, 11F27.
	
	\section{Introduction}
	
	For $q=e^{2\pi i \tau}$, $Im(\tau)>0$, we define Dedekind's eta function $\eta(\tau)$ by 
	\begin{equation*}
		\eta(\tau)= q^{1/24}\prod\limits_{n=1}^{\infty}(1-q^n).
	\end{equation*}
	%It was introduced by Dedekind in 1877 \cite{}. 
	This function is useful, since many modular functions can be expressed using it and many combinatorial generating functions can be constructed from it. For example, let $r_k(n)$ be the number of representations of an integer $n$ as a sum of $k$ squares; then we have 
	\begin{equation*}
		\sum\limits_{n=0}^{\infty} r_k(n) q^n= \dfrac{\eta^{5k}(2\tau)}{\eta^{2k}(\tau)\eta^{2k}(4\tau)}.
	\end{equation*}
	
	Eta quotients have connections with many other areas of mathematics. For the connection with the partition function, see \cite{ono}. In \cite{oliver}, Oliver expressed the eta function as a certain kind of infinite series known as a theta function. Using the theory of modular forms, he classified all eta quotients which are theta functions. For example,
	
\begin{align*}
	\dfrac{\eta(\tau)\eta(32\tau)}{\eta(16\tau)} =\ds\sum\limits_{n=1}^{\infty}  \legendre{2}{n} q^{n^2}
	\intertext{and}
	\dfrac{\et{3}^2 \et{2}^2}{\et{6}} = \non \legendre{n}{3} n q^{n^2}.
\end{align*}

	 In \cite{oliver}, Oliver expressed 23 eta quotient identities in terms theta functions. The aim of this article is to give an alternative proof for all the identities established by Oliver. In Section 2, we give our proof. We close this section by recalling certain definitions and known results, required to prove our main results.
	 
	 For an odd prime $p$, Legendre's symbol is defined by 
	 \begin{equation*}
	 	\legendre{n}{p} = \begin{cases}
	 		0,~~~ \text{if}~ p \mid n\\
	 		1, ~~~\text{if $n$ is a quadratic residue modulo $p$}\\
	 		-1, ~\text{if $n$ is a non quadratic residue modulo $p$}
	 	\end{cases} .
	 \end{equation*}
	 The Kronecker's symbol $\legendre{n}{m}$ is defined by 
	\begin{equation*}
		\legendre{n}{m} = \begin{cases}
			1,~~~ \text{if}~ m=1\\
			0, ~~~\text{if $m$ is a prime dividing $n$}\\
			\text{Legendre symbol}, ~\text{if $m$ is an odd prime}
		\end{cases}.
	\end{equation*}
	 Also, we have
	 \begin{equation}\label{n/2}
	 	\legendre{n}{2} = \begin{cases}
	 			0,~~~ \text{if $n$ is even}\\
	 			1, ~~~\text{if $n$ is odd and}~ n \equiv \pm 1(\mathrm{mod}~ 8)\\
	 			-1, ~\text{if $n$ is odd and}~ n \equiv \pm 3(\mathrm{mod}~ 8)
	 	\end{cases}.
	 \end{equation}
	In general 
	\begin{equation*}
	\legendre{n}{m} = \prod\limits_{i=1}^{s}\legendre{n}{p_i},
	\end{equation*} where $m=\prod\limits_{i=1}^{s}p_i$ is a prime factorization of $m$. One can easily show that
	 \begin{gather*}
	 \legendre{a}{bc} = \legendre{a}{b} \legendre{a}{c} ~~\text{and} ~~ \legendre{ab}{c} = \legendre{a}{c} \legendre{b}{c}.
	 \end{gather*}
	For more details, one can refer to \cite[pp. 30--39, p. 77]{dickson} and \cite[pp. 509--532]{koshy}.

On page 197 of his second notebook \cite{Ramanujan}, S. Ramanujan defined his general theta function $f(a, b)$ as follows: 
	\begin{equation*}
		f(a,b): = \sum\limits_{n=-\infty}^{\infty} a^{\frac{n(n+1)}{2}} b^{\frac{n(n-1)}{2}},~ |ab|<1.
	\end{equation*}

	\noindent Ramanujan considered the three following special cases of $f(a,b)$ as follows :
	\begin{equation}\label{phi}
		\varphi(q): = f(q,q) = \sum\limits_{n=-\infty}^{\infty}q^{n^2}=(-q;q^2)^2_{\infty}(q^2;q^2)_{\infty},
	\end{equation}
	\begin{equation}\label{psi}
		\psi(q): = f(q,q^3) = \sum\limits_{n=0}^{\infty}q^{\frac{n(n+1)}{2}} = \frac{(q^2;q^2)_{\infty}}{(q;q^2)_{\infty}},
	\end{equation}
	and
	\begin{equation}\label{f}
		f(-q) := f(-q,-q^2) = \sum\limits_{n=-\infty}^{\infty}(-1)^nq^{\frac{n(3n+1)}{2}}=(q;q)_{\infty}.
	\end{equation}	
	
	\ni Ramanujan also defined
	\begin{equation*}
		\chi(-q): = (q;q^2)_{\infty}.
	\end{equation*}
	
Now, we shall recall the well-known relations between the above Ramanujan's theta functions and the eta function.

	\begin{align}
				\varphi(q) =& \dfrac{\et{2}^5}{\et{}^2\et{4}^2}, \label{eq1.1} \\
		\varphi(-q) =& \dfrac{\et{}^2}{\et{2}}, \label{eq1.2}\\
			q^{\frac{1}{8}}	\psi(q) =& \dfrac{\et{2}^2}{\et{}}, \label{eq1.3}\\
			q^{\frac{1}{8}}	\psi(-q) =& \dfrac{\et{}\et{4}}{\et{2}}, \label{eq1.4}\\
			q^{\frac{1}{24}} f(-q) =& \et{}, \label{eq1.5}\\
			\intertext{and}
			q^{\frac{1}{24}} f(q) =& \dfrac{\et{2}^3}{\et{}\et{4}}. \label{eq1.6}
		\end{align}

		The following known identities will be used in our proof.
				\begin{align}
			f^3(-q)=& \szn (-1)^n (2n+1) q^{\frac{n(n+1)}{2}}, \label{eq1.7}
		\end{align}
		\begin{align}
			qf(-q^{24}) =& \sinn (-1)^n q^{(6n+1)^2}, \label{eq1.8}
		\end{align}
		\begin{align}
			\varphi(q)- \varphi(q^9) =& 2q\dfrac{f^2(-q^6)f(-q^9)f(-q^{36})}{f(-q^3)f(-q^{12})f(-q^{18})}, \label{eq1.9}
		\end{align}
		\begin{align}
				  \psi(q)- q\psi(q^9)  =&q^{\frac{1}{8}} \dfrac{f(-q^6)f^2(-q^9)}{f(-q^3)f(-q^{18})}, \label{eq1.10}
				\end{align}
				\begin{align}
				3\varphi(q^9) - \varphi(q) =& 2\dfrac{f(-q)f(-q^4)f^2(-q^6)}{f(-q^2)f(-q^3)f(-q^{12})}, \label{eq1.11}\\
				 \psi(q)-3 q\psi(q^9)  =&  q^{\frac{1}{8}} \dfrac{f^2(-q)f(-q^6)}{f(-q^2)f(-q^3)}, \label{eq1.12}\\
			q^{\frac{1}{3}}	\dfrac{f^2(-q^4)f^2(-q)}{f(-q^2)} =& \sinn (3n+1) q^{3n^2+2n},\label{eq1.13}\\
				\intertext{and}
			q^{\frac{1}{24}}	\dfrac{f^5(-q)}{f^2(-q^2)} =& \sinn (6n+1) q^{\frac{3n^2+n}{2}}. \label{eq1.14}
	\end{align}
A proof of \eqref{eq1.7} can be found in \cite[p.~14]{spirit}. The identities \eqref{eq1.9} and \eqref{eq1.10} are recorded in \cite[p.~49, (i) and (ii)]{Berndt3}, respectively. The identities \eqref{eq1.13} and \eqref{eq1.14} are found in \cite[pp.~114--115, Entry 8(x) and (ix)]{Berndt3}, respectively, while the identity \eqref{eq1.12} appears in \cite[p.~349, Entry 2(ii)]{Berndt3}. The identity \eqref{eq1.8} can be obtained by changing $q$ to $q^{24}$ in \eqref{f}. The identity \eqref{eq1.11} is equivalent to \cite[p.~352 Entry 2(i)]{Berndt3}.

\section{Main results}
\begin{thm} We have
	\begin{align}
		\dfrac{\et{2}^5}{\et{}^2 \et{4}^2} &= \sinn q^{n^2}, \label{eq2.1}\\
		\dfrac{\et{8}\et{32}}{\et{16}} &= \non \legendre{2}{n} q^{n^2},\label{eq2.2}\\
		\dfrac{{\eta(16\tau)}^2}{\eta(8\tau)}&=\non\legendre{n}{2}^2 q^{n^2}, \label{eq2.3} \\
		\dfrac{\et{6}^2 \et{9} \et{36}}{\et{3}\et{12}\et{18}} &= \non \legendre{n}{3}^2 q^{n^2}, \label{eq2.4} \\
		\et{24} &= \non \legendre{12}{n} q^{n^2}, \label{eq2.5} \\ 
		\dfrac{\et{48}^3}{\et{24}\et{96}} &= \non \legendre{24}{n} q^{n^2}, \label{eq2.6} \\ 
		\dfrac{\et{48}\et{72}^2}{\et{24}\et{144}} &= \non \legendre{n}{6}^2 q^{n^2}, \label{eq2.7} \\
		\intertext{and}
		\dfrac{\et{24}\et{96}\et{144}^5}{\et{48}^2\et{72}^2 \et{288}} &= \non \legendre{18}{n} q^{n^2}. \label{eq2.8}
	\end{align}
\end{thm}
\begin{proof}[Proof of \eqref{eq2.2}] Consider \begin{equation}\label{eq2.9}
	\ds\non \legendre{2}{n}q^{n^2} = \sum\limits_{k=1}^{8}\szn \legendre{2}{8n+k} q^{(8n+k)^2}.
	\end{equation}
	From \cite[p.~492]{koshy}, we have
	 \begin{equation*}
		\legendre{2}{n} = \begin{cases}
			0,~~~ \text{if $n$ is even}\\
			1, ~~~\text{if $n$ is odd and}~ n \equiv \pm 1(\mathrm{mod}~ 8)\\
			-1, ~\text{if $n$ is odd and}~ n \equiv \pm 3(\mathrm{mod}~ 8)
		\end{cases}.
	\end{equation*}
	Using the above definition in \eqref{eq2.9}, we obtain 
	\begin{equation*}
		\non \legendre{2}{n} q^{n^2}=\szn q^{(8n+1)^2} - \szn q^{(8n+3)^2} - \szn q^{(8n+5)^2} + \szn q^{(8n+7)^2}.
	\end{equation*}
	The above equation is equivalent to
	\begin{align*}
		\non \legendre{2}{n} q^{n^2}& = \szn (-1)^{\frac{n(n+1)}{2}} q^{(2n+1)^2}\\
		&=q \szn  (-1)^{\frac{n(n+1)}{2}} \left(q^8 \right)^{\frac{n(n+1)}{2}} .
	\end{align*}
Replacing $q$ by $-q$ in \eqref{psi} and then changing $q$ to $q^8$, employing the resulting identity in the above, we find that
	\begin{equation*}
	\non \legendre{2}{n} q^{n^2} = q\psi(-q^8).
	\end{equation*}
Expressing  $\psi(-q^8)$ in terms of eta quotients, we complete the proof of \eqref{eq2.2}.\\
	
\ni	\textit{Proof of \eqref{eq2.3}.}
Employing \eqref{n/2} in the right hand side of \eqref{eq2.3}, we obtain
\begin{align*}
	\non \legendre{n}{2}^2 q^{n^2} &= \szn q^{(2n+1)^2}\\
	&= q \szn   \left(q^8 \right)^{\frac{n(n+1)}{2}}.
\end{align*}
Changing $q$ to $q^8$ in \eqref{psi}, then expressing the resulting identity in terms of eta quotients in the above equation, we obtain the required identity.\\

\ni	\textit{Proof of \eqref{eq2.4}.} We have
\begin{equation}\label{n/3}
	\legendre{n}{3} = \begin{cases}
		1, ~~~~\text{if}~n\equiv 1 (\textrm{mod}~3),\\
		-1, ~~\text{if}~n\equiv 2(\textrm{mod}~3),\\
		0, ~~\text{if}~n\equiv 0 (\textrm{mod}~3).
	\end{cases}
\end{equation}
Using the above in the right hand side of \eqref{eq2.4}, we have
\begin{equation*}
	\non \legendre{n}{3}^2 q^{n^2} = \szn q^{(3n+1)^2} + \szn q^{(3n+2)^2}.
\end{equation*}
The above equation implies
\begin{equation*}
	\non \legendre{n}{3}^2 q^{n^2}= \non q^{n^2} - \non q^{(3n)^2}.
\end{equation*}
Employing the definition of $\varphi(q)$ in the above, we obtain
\begin{equation*}
\non \legendre{n}{3}^2 q^{n^2} = \dfrac{\varphi(q)-\varphi(q^9)}{2}.
\end{equation*}
Employing \eqref{eq1.9} in the above, and expressing the resulting identity in terms of eta functions, we obtain the required identity.\\

\ni	\textit{Proof of \eqref{eq2.5}.} By the definition of $\legendre{a}{n}$, we have
\begin{equation*}
	\legendre{12}{n} = \begin{cases}
		1, ~~~\text{if}~ n\equiv 1, 11 (\mathrm{mod~12}),\\
		-1, ~\text{if}~ n\equiv 5, 7 (\mathrm{mod~12}),\\
		0, ~~~\text{if} ~gcd(12,n)>1.
	\end{cases}
\end{equation*}
Using the above in the right hand side of \eqref{eq2.5}, we obtain
\begin{equation*}
	\non \legendre{12}{n} q^{n^2} = \szn q^{(12n+1)^2} - \szn q^{(12n+5)^2} -\szn q^{(12n+7)^2}+\szn q^{(12n+11)^2}.
\end{equation*}
The above identity is equivalent to 
\begin{equation*}
	\non \legendre{12}{n} q^{n^2} = \szn (-1)^n q^{(6n+1)^2} - \szn (-1)^n q^{(6n+5)^2}.
\end{equation*}
The above identity can also be written as
\begin{align*}
\non \legendre{12}{n} q^{n^2} = \sinn (-1)^n q^{(6n+1)^2}.\\ 
\end{align*}
Employing \eqref{eq1.8} in the above, and then expressing the resulting identity in terms of eta functions, we complete the proof of \eqref{eq2.5}.\\

\ni	\textit{Proof of \eqref{eq2.6}.} By the definition of $\legendre{a}{n}$, we have
\begin{equation*}
	\legendre{24}{n} = \begin{cases}
		1, ~~~\text{if}~ n\equiv \pm1~\text{or}~ \pm 5 (\mathrm{mod~24}),\\
		-1, ~\text{if}~ n\equiv \pm 7 ~\text{or}~\pm 11 (\mathrm{mod~24}),\\
		0, ~~~\text{if} ~gcd(24,n)>1.
	\end{cases}
\end{equation*}
Using the above in the right hand side of \eqref{eq2.6}, we obtain
\begin{equation*}
	\non \legendre{24}{n} q^{n^2} = \sum\limits_{k\in\left\lbrace 1, 5, 19,23\right\rbrace }^{} \szn q^{(24n+k)^2} - \sum\limits_{k\in\left\lbrace 7, 11, 13, 17\right\rbrace }^{} \szn q^{(24n+k)^2}.
\end{equation*}
The above identities can be rewritten as
\begin{equation*}
	\non \legendre{24}{n} q^{n^2} = \sinn q^{(24n+1)^2}- \sinn q^{(24n+7)^2} - \sinn q^{(24n+13)^2} + \sinn q^{(24n+19)^2}.
\end{equation*}
The above identity is equivalent to 
\begin{align*}
	\non \legendre{24}{n} q^{n^2} &=\sinn(-1)^{\frac{n(n+1)}{2}} q^{(6n+1)^2} \\&=q \sinn (-1)^{\frac{n(n+1)}{2}}\left( q^{24}\right) ^{\frac{n(3n+1)}{2}}.
\end{align*}
Changing $q$ to $-q$ in \eqref{f}, and then replacing $q$ by $q^{24}$ and then employing the resulting identity in the above, we obtain
\begin{align*}
	\non \legendre{24}{n} q^{n^2} =q f(q^{24}).
\end{align*}
Expressing $f(q^{24})$ in terms of eta quotients using \eqref{eq1.6}, we obtain the required result.\\

\ni	\textit{Proof of \eqref{eq2.7}.} We have
\begin{equation}\label{n/6}
	\legendre{n}{6} = \begin{cases}
		1, ~~~~\text{if}~n\equiv 1 (\textrm{mod}~6),\\
		-1, ~~\text{if}~n\equiv 5(\textrm{mod}~6),\\
		0, ~~\text{if}~ gcd(n,6)>1.
	\end{cases}
\end{equation}
Using the above in right hand side of \eqref{eq2.7}, we obtain
\begin{align*}
	\non \legendre{n}{6}^2 q^{n^2} = \szn q^{(6n+1)^2} + \szn q^{(6n+5)^2}.
\end{align*}
The above identity is equivalent to 
\begin{align*}
	\non \legendre{n}{6}^2 q^{n^2} &= \szn q^{(2n+1)^2} - \szn q^{(6n+3)^2}\\ &= q \szn   \left(q^8 \right)^{\frac{n(n+1)}{2}}- q^9 \szn   \left(q^{72} \right)^{\frac{n(n+1)}{2}}.
\end{align*}
Which is equivalent to 
\begin{equation*}
\non \legendre{n}{6}^2 q^{n^2} = q\psi(q^8)-q^9 \psi(q^{72}).
\end{equation*}
Changing $q$ to $q^8$ in \eqref{eq1.10} and employing the resulting identity in terms of eta functions in the above, we obtain the required result.\\
The identity \eqref{eqthm2.21} is equivalent to \eqref{phi}.
\end{proof}

		\begin{thm} We have
			\begin{align}
				\et{8}^3 &= \non \legendre{-4}{n} n q^{n^2},\label{eqthm2.21} \\
				\dfrac{\et{16}^9}{\et{8}^3 \et{32}^3}&= \non \legendre{-2}{n} n q^{n^2}, \label{eqthm2.22}\\
					\dfrac{\et{3}^2\et{12}^2}{ \et{6}}&= \non \legendre{n}{3} n q^{n^2}, \label{eqthm2.23}\\
						\dfrac{\et{48}^{13}}{\et{24}^5 \et{96}^5}&= \non \legendre{-6}{n} n q^{n^2},\label{eqthm2.24}\\
						\intertext{and}
						\dfrac{\et{24}^5}{\et{48}^2} &= \non \legendre{n}{12} n q^{n^2}. \label{eqthm2.25}
			\end{align}
		\end{thm}
		\begin{proof}[Proof of \eqref{eqthm2.21}]
		By the definition of $\legendre{a}{n}$, we have \begin{equation*}
			\legendre{-4}{n} = \begin{cases}
				0, ~~~~~~~\text{if}~ n \equiv 0(\mathrm{mod~2}),\\
				(-1)^k,~\text{if}~ n \equiv 1(\mathrm{mod~2}), \text{where}~ n=2k+1.
			\end{cases}
		\end{equation*}
	Using the above in the right-hand side of \eqref{eqthm2.21}, we obtain
		\begin{align*}
			\non\legendre{-4}{n} n q^{n^2}&= \szn (-1)^n (2n+1) q^{(2n+1)^2}\\
			&= q \szn (-1)^n (2n+1)   \left(q^8 \right)^{\frac{n(n+1)}{2}}.
		\end{align*}
		Changing $q$ to $q^8$ in \eqref{eq1.7} and then employing the resulting identity in terms of eta functions in the above, we complete the proof of \eqref{eqthm2.21}.\\
		
	\ni	\textit{Proof of \eqref{eqthm2.22}.} By the definition of $\legendre{a}{n}$, we have
	\begin{equation*}
		\legendre{-2}{n} = \begin{cases}
		1, ~~~\text{if}~ n \equiv 1 ~\text{or}~ 3(\mathrm{mod~8}),\\
		-1,~\text{if}~ n \equiv 5 ~\text{or}~7(\mathrm{mod~2}),\\ 
		0,~~~ \text{if }~ n~ \text{is even}.
	\end{cases}
	\end{equation*}
	Employing the above in the right-hand side of \eqref{eqthm2.22}, we have
	\begin{align*}
		\non \legendre{-2}{n} n q^{n^2} = &\szn (8n+1) q^{(8n+1)^2} + \szn (8n+3) q^{(8n+3)^2} \\&- \szn (8n+5) q^{(8n+5)^2} - \szn (8n+7) q^{(8n+7)^2}.
	\end{align*}
	The above identity is equivalent to 
	\begin{align*}
		\non \legendre{-2}{n} n q^{n^2}= &	\szn (-1)^{\frac{n(n+1)}{2}}(2n+1) q^{(2n+1)^2}\\
		&=q \szn  (-1)^{\frac{n(n+1)}{2}} (2n+1)\left(q^8 \right)^{\frac{n(n+1)}{2}} .
	\end{align*}
Replacing $q$ by $-q$ in \eqref{eq1.7} and then changing $q$ to $q^8$ and employing the resulting identity in terms of eta functions in the above, we obtain the required result.\\

		\ni	\textit{Proof of \eqref{eqthm2.23}.} Employing \eqref{n/3} in the right-hand side of \eqref{eqthm2.23}, we obtain
		\begin{align*}
			\non \legendre{n}{3} n q^{n^2}	&= \szn (3n+1)q^{(3n+1)^2} -\szn (3n+2)q^{(3n+2)^2} \\
			&= \sinn (3n+1)q^{(3n+1)^2}\\
			&= q\sinn (3n+1)\left( q^3\right) ^{(3n^2+2n)}. 
		\end{align*}
		Changing $q$ to $q^3$ in \eqref{eq1.13} and employing the resulting identity in terms of eta functions in the above,  we obtain the required result.\\
		
			\ni	\textit{Proof of \eqref{eqthm2.24}.} 	By the definition of $\legendre{a}{n}$, we have \begin{equation*}
				\legendre{-6}{n} = \begin{cases}
					1, ~~~\text{if}~ n \equiv 1, 5, 7 ~\text{or}~ 11(\mathrm{mod~24}),\\
					-1,~\text{if}~ n \equiv 13, 17, 19 ~\text{or}~23(\mathrm{mod~2}),\\ 
					0,~~~ \text{if }~ gcd(n,6)>1.
				\end{cases}
			\end{equation*}
			Using the above in the right-hand side of \eqref{eqthm2.24}, we have
			\begin{align*}
				\non \legendre{-6}{n} n q^{n^2}& = \sum\limits_{\substack{k=1\\k\neq 2, 5}}^{6} \szn (24n+2k-1) q^{(24n+2k-1)^2}\\& - \sum\limits_{\substack{k=7\\k\neq 8, 11}}^{12} \szn (24n+2k-1) q^{(24n+2k-1)^2}.				
			\end{align*}
			The above equation implies
			\begin{align*}
			\non \legendre{-6}{n} n q^{n^2}& = \sinn (24n+1) q^{(24n+1)^2} + \sinn (24n+7) q^{(24n+7)^2} \\ &- \sinn (24n+13) q^{(24n+13)^2} - \sinn (24n+19) q^{(24n+19)^2}.
			\end{align*}
			Which is equivalent to 
			\begin{align*}
			\non \legendre{-6}{n} n q^{n^2} &=\sinn (-1)^{\frac{n(3n+1)}{2}}(6n+1) q^{(6n+1)^2}\\
			&=q \sinn  (-1)^{\frac{n(3n+1)}{2}} (6n+1)\left(q^{12} \right)^{3n^2+n} .
			\end{align*}
			Replacing $q$ by $-q$ in \eqref{eq1.14} and then changing $q$ to $q^{24}$ and employing the resulting identity in terms of eta functions in the above, we obtain the required result.\\
			
			\ni	\textit{Proof of \eqref{eqthm2.25}.} 	By the definition of $\legendre{a}{n}$, we have
			\begin{equation}\label{n/12}
				\legendre{n}{12} = \begin{cases}
					1, ~~~\text{if}~ n\equiv 1(\mathrm{mod~6})\\
						-1, ~\text{if}~ n\equiv 5(\mathrm{mod~6})
				\end{cases}.
			\end{equation}
			Employing the above definition in the right-hand side of \eqref{eqthm2.25}, we obtain 
			\begin{align*}
				\non \legendre{n}{12} n q^{n^2}&=  \szn (6n+1) q^{(6n+1)^2} -  \szn (6n+5) q^{(6n+5)^2}\\ &=  \sinn (6n+1) q^{(6n+1)^2}\\
				&= q  \sinn (6n+1)\left(  q^{12}\right) ^{(3n^2+n)}.
			\end{align*}
			Employing \eqref{eq1.14} in terms of eta functions in the above, we obtain the required result.
		\end{proof}
		\begin{thm} We have
			\begin{align}
			\dfrac{\et{}^2}{\et{2}}	 &= \sinn \left( 1-2\legendre{n}{2}^2\right) q^{n^2}=\sinn(-1)^{n}q^{n^2},\label{eqthm2.31}\\	
			\dfrac{\et{}\et{4}\et{6}^2}{\et{2}\et{3}\et{12}}&=\sinn \left( 1-\dfrac{3}{2}\legendre{n}{3}^2\right) q^{n^2}, \label{eqthm2.32}\\		
			\dfrac{\et{2}^2\et{3}}{\et{}\et{6}}&= \sinn\left(1-2\legendre{n}{2}^2-\dfrac{3}{2}\legendre{n}{3}^2+3\legendre{n}{6}^2 \right) q^{n^2}, \label{eqthm2.33}\\
				\dfrac{\et{8}^5}{\et{4}^2\et{16}^2}&= \sinn\left( 1-\legendre{n}{2}^2\right)q^{n^2},\label{eqthm2.34} \\
				\dfrac{\et{9}^2}{\et{18}}&= \sinn\left( 1-2\legendre{n}{2}-\legendre{n}{3}^2+2\legendre{n}{6}^2\right)q^{n^2},\label{eqthm2.35} \\
				\dfrac{\et{18}^5}{\et{9}^2\et{36}^2}&=\sinn\left( 1-\legendre{n}{3}^2\right) q^{n^2}, \label{eqthm2.36}\\
				\dfrac{\et{4}\et{16}\et{24}^2}{\et{8}\et{12}\et{48}}&= \sinn\left( 1-\legendre{n}{2}^2-\dfrac{3}{2}\legendre{n}{3}^2+\dfrac{3}{2}\legendre{n}{6}^2\right) q^{n^2}, \label{eqthm2.37}\\
				\dfrac{\et{72}^5}{\et{36}^2\et{144}^2}&= \sinn \left(1-\legendre{n}{2}^2-\legendre{n}{3}^2+\legendre{n}{6}^2 \right) q^{n^2},\label{eqthm2.38}\\
				\dfrac{\et{3}\et{18}^2}{\et{6}\et{9}} &= \non\left(2\legendre{n}{6}^2 - \legendre{n}{3}^2 \right) q^{n^2}, \label{eqthm2.39}\\
				\intertext{and}
				\dfrac{\et{8}^2\et{48}}{\et{16}\et{24}} &= \non\left(3\legendre{n}{6}^2-2\legendre{n}{2}^2 \right) q^{n^2}. \label{eqthm2.310}
			\end{align}
		\end{thm}
		\begin{proof}[Proof of \eqref{eqthm2.31}]
		Employing \eqref{n/2} in the right-hand side of \eqref{eqthm2.31}, we obtain 
		\begin{align*}
		\sinn \left( 1-2\legendre{n}{2}^2\right) q^{n^2} &= \sinn q^{(2n)^2} - \sinn q^{(2n+1)^2}\\ &= \sinn(-1)^n q^{n^2}.
		\end{align*}
		Expressing $\varphi(-q)$ in terms of eta quotients in the above, we obtain the required result.\\
		
			\ni	\textit{Proof of \eqref{eqthm2.32}.} Employing \eqref{n/3} in the right-hand side of \eqref{eqthm2.32}, we obtain
			\begin{align*}
				\sinn \left( 1-\dfrac{3}{2}\legendre{n}{3}^2\right) q^{n^2} &= -\frac{1}{2}\sinn q^{(3n+1)^2} -\frac{1}{2}\sinn q^{(3n+2)^2} + \sinn q^{(3n)^2} \\
				&= -\frac{1}{2}\sinn q^{n^2} + \frac{3}{2}\sinn q^{(3n)^2}.
			\end{align*}
			Which implies
			\begin{equation*}
				\sinn \left( 1-\dfrac{3}{2}\legendre{n}{3}^2\right) q^{n^2} = \dfrac{-\varphi(q)+3\varphi(q^9)}{2}.
			\end{equation*}
			Employing \eqref{eq1.11} in terms of eta functions in the above, we obtain the required result.\\
			
				\ni	\textit{Proof of \eqref{eqthm2.33}.} Employing \eqref{n/2}, \eqref{n/3} and \eqref{n/6} in the right-hand side of \eqref{eqthm2.33}, we find that 
				\begin{align*}
					\sinn&\left(1-2\legendre{n}{2}^2-\dfrac{3}{2}\legendre{n}{3}^2+3\legendre{n}{6}^2 \right) q^{n^2} \\&= \frac{1}{2}\sinn q^{(6n+1)^2} - \frac{1}{2}\sinn q^{(6n+2)^2} -\sinn q^{(6n+3)^2} \\&- \frac{1}{2}\sinn q^{(6n+4)^2} + \frac{1}{2}\sinn q^{(6n+5 )^2} + \sinn q^{(6n)^2}\\
					&=-\frac{1}{2}\sinn(-1)^n q^{n^2} + \frac{3}{2}\sinn(-1)^n q^{(3n)^2}.
				\end{align*}
				Which implies
				\begin{align*}
				\sinn&\left(1-2\legendre{n}{2}^2-\dfrac{3}{2}\legendre{n}{3}^2+3\legendre{n}{6}^2 \right) q^{n^2} = \dfrac{-\varphi(-q)+3\varphi(-q^9)}{2}.
				\end{align*}
				Changing $q$ to $-q$ in \eqref{eq1.11} and then employing the resulting identity in terms of eta functions in the above, we complete the proof of \eqref{eqthm2.33}.\\
				
				\ni	\textit{Proof of \eqref{eqthm2.34}.} Employing \eqref{n/2} in the right-hand side of \eqref{eqthm2.34}, we obtain
				\begin{align*}
					\sinn\left( 1-\legendre{n}{2}^2\right)q^{n^2} &= \sinn q^{(2n)^2}\\ &= \varphi(q^4).
				\end{align*}
				Expressing $\varphi(q^4)$ in terms of eta quotients, the result follows.\\
				
				\ni	\textit{Proof of \eqref{eqthm2.35}.} Employing \eqref{n/2}, \eqref{n/3} and \eqref{n/6} in the right-hand side of \eqref{eqthm2.35}, we find that 
				\begin{align*}
				\sinn\left( 1-2\legendre{n}{2}-\legendre{n}{3}^2+2\legendre{n}{6}^2\right)q^{n^2}& = -\sinn q^{(6n+3)^2} + \sinn q^{(6n)^2}\\
				&= \sinn (-1)^n q^{(3n)^2}\\
				&=\varphi(-q^9).
				\end{align*}
				Expressing $\varphi(-q^9)$ in terms of eta quotients, the result follows.\\
				
					\ni	\textit{Proof of \eqref{eqthm2.36}.} Employing \eqref{n/3} in the right-hand side of \eqref{eqthm2.36}, we obtain
					\begin{align*}
						\sinn\left( 1-\legendre{n}{3}^2\right) q^{n^2} &= \sinn q^{(3n)^2}\\ 
						&= \varphi(q^9).
					\end{align*}
					Expressing $\varphi(q^9)$ in terms of eta quotients, the result follows.\\
					
						\ni	\textit{Proof of \eqref{eqthm2.37}.} Employing \eqref{n/2}, \eqref{n/3} and \eqref{n/6} in the right-hand side of \eqref{eqthm2.37}, we find that
			\begin{align*}
			\sinn&\left( 1-\legendre{n}{2}^2-\dfrac{3}{2}\legendre{n}{3}^2+\dfrac{3}{2}\legendre{n}{6}^2\right) q^{n^2}\\& = -\frac{1}{2} \sinn q^{(6n+2)^2} - \frac{1}{2} \sinn q^{(6n+4)^2} + \sinn q^{(6n)^2}\\ 
			&= -\frac{1}{2} \sinn q^{(2n)^2} + \frac{3}{2} \sinn q^{(6n)^2}.
				\end{align*} 
				The above implies
				\begin{equation*}
				\sinn\left( 1-\legendre{n}{2}^2-\dfrac{3}{2}\legendre{n}{3}^2+\dfrac{3}{2}\legendre{n}{6}^2\right) q^{n^2} = \dfrac{-\varphi(q^4)+3\varphi(q^{36})}{2}.
				\end{equation*}
				Changing $q$ to $q^4$ in \eqref{eq1.11} and then using the resulting identity in terms of eta functions in the above, we obtain the required result.\\
				
			\ni	\textit{Proof of \eqref{eqthm2.38}.} Employing \eqref{n/2}, \eqref{n/3} and \eqref{n/6} in the right-hand side of \eqref{eqthm2.38}, we find that
			\begin{align*}
				\sinn \left(1-\legendre{n}{2}^2-\legendre{n}{3}^2+\legendre{n}{6}^2 \right) q^{n^2} &= \sinn q^{(6n)^2}\\
				&= \varphi(q^{36}).
			\end{align*}
		Expressing $\varphi(q^{36})$ in terms of eta quotients, the result follows.\\
			
			\ni	\textit{Proof of \eqref{eqthm2.39}.} Employing \eqref{n/3} and \eqref{n/6} in the right-hand side of \eqref{eqthm2.39}, we obtain
			\begin{align*}
			\non\left(2\legendre{n}{6}^2 - \legendre{n}{3}^2 \right) q^{n^2} &= \szn q^{(6n+1)^2} - \szn q^{(6n+2)^2} - \szn q^{(6n+4)^2} + \szn q^{(6n+5)^2}\\
			&=-\szn (-1)^n q^{n^2} - \szn  q^{(6n+3)^2} + \szn q^{(6n)^2}\\
			&= -\szn (-1)^n q^{n^2} + \szn (-1)^n q^{(3n)^2}.
			\end{align*}
			Which implies
			\begin{align*}
					\non\left(2\legendre{n}{6}^2 - \legendre{n}{3}^2 \right) q^{n^2} = \dfrac{-\varphi(-q) + \varphi(-q^9)}{2}.
			\end{align*}
			Changing $q$ to $-q$ in \eqref{eq1.9} and then employing the resulting identity in terms of eta functions in the above, we obtain the required result.\\
			
			\ni	\textit{Proof of \eqref{eqthm2.310}.} Employing \eqref{n/2} and \eqref{n/6} in the right-hand side of \eqref{eqthm2.310}, we obtain
			\begin{align*}
				\non\left(3\legendre{n}{6}^2-2\legendre{n}{2}^2 \right) q^{n^2} &= \szn q^{(6n+1)^2} - 2\szn q^{(6n+3)^2} + \szn q^{(6n+5)^2}\\
				&= \szn q^{(2n+1)^2} - 3 \szn \left(q^9 \right)^{(2n+1)^2} \\
		&	=q \szn   \left(q^8 \right)^{\frac{n(n+1)}{2}}-3 q^9 \szn   \left(q^{72} \right)^{\frac{n(n+1)}{2}}.
			\end{align*}
			Which is equivalent to 
			\begin{equation*}
				\non\left(3\legendre{n}{6}^2-2\legendre{n}{2}^2 \right) q^{n^2} = q\psi(q^8)- 3 q^9 \psi(q^{72}).
			\end{equation*}
			Changing $q$ to $q^8$ in \eqref{eq1.10} and then employing the resulting identity in the above, we obtain the required result.
		\end{proof}
		
		\begin{thm} We have
			\begin{equation}\label{eqthm2.41}
				\dfrac{\et{6}^5}{\et{3}^2} = \non \left(2\legendre{n}{12} - \legendre{n}{3} \right) n q^{n^2}.
			\end{equation}
				\end{thm}
			\begin{proof}[Proof of \eqref{eqthm2.41}.] Employing \eqref{n/3} and \eqref{n/12} in the right-hand side of \eqref{eqthm2.41}, we obtain
				\begin{align*}
					\non \left(2\legendre{n}{12} - \legendre{n}{3} \right) n q^{n^2} =& \szn (6n+1) q^{(6n+1)^2} + \szn (6n+2) q^{(6n+2)^2}\\ & - \szn (6n+4) q^{(6n+4)^2}- \szn (6n+5) q^{(6n+5)^2}\\
					=& \szn (-1)^n (3n+1) q^{(3n+1)^2} + \szn (-1)^n (3n+2) q^{(3n+2)^2}.
				\end{align*}
				Which is equivalent to
				\begin{align*}
						\non \left(2\legendre{n}{12} - \legendre{n}{3} \right) n q^{n^2} &=  \sinn (-1)^n (3n+1) q^{(3n+1)^2}\\
						&= q\sinn (-1)^n (3n+1) \left(q^3 \right) ^{3n^2+2n}.
				\end{align*}
			Replacing $q$ by $-q$ in \eqref{eq1.13} and then changing $q$ to $q^3$ and employing the resulting identity in terms of eta functions in the above, we obtain the required result.
			\end{proof}
			
			\begin{thm}
				We have
				\begin{align}
					\dfrac{\et{6}\et{9}^2}{\et{3}\et{18}} &= \szn \legendre{n+2}{3}^2 q^{\frac{n(n+1)}{2}} \label{eqthm2.51}\\
					\intertext{and}
					\dfrac{\et{}^2\et{6}}{\et{2}\et{3}} &= \szn \left[  3\legendre{n+2}{3}^2-2\right]  q^{\frac{n(n+1)}{2}}. \label{eqthm2.52}
				\end{align}
			\end{thm}
			\begin{proof}[Proof of \eqref{eqthm2.51} ]
			By employing \eqref{n/3} in the right-hand side of \eqref{eqthm2.51}, we obtain
			\begin{align*}
				\szn \legendre{n+2}{3}^2 q^{\frac{n(n+1)}{2}} &=\szn q^{\frac{(3n)(3n+1)}{2}}+ \szn q^{\frac{(3n+2)(3n+3)}{2}} \\
				&=\szn q^{\frac{n(n+1)}{2}} - \szn q^{\frac{(3n+1)(3n+2)}{2}}\\
				& = \szn q^{\frac{n(n+1)}{2}} - q \szn \left(q^9 \right)^{\frac{n(n+1)}{2}}.
			\end{align*}
			Which is equivalent to 
			\begin{equation*}
					\szn \legendre{n+2}{3}^2 q^{\frac{n(n+1)}{2}} = \psi(q)-q\psi(q^9).
			\end{equation*}
			Employing \eqref{eq1.10} in terms of eta functions in the above, we obtain the required identity.\\
			
			\ni	\textit{Proof of \eqref{eqthm2.52}.} Employing \eqref{n/3} in the right-hand side of the above, we obtain
			\begin{align*}
			\szn \left[  3\legendre{n+2}{3}^2-2\right]  q^{\frac{n(n+1)}{2}} =& \szn q^{\frac{(3n)(3n+1)}{2}} -2\szn q^{\frac{(3n+1)(3n+2)}{2}} \\ & + \szn q^{\frac{(3n+2)(3n+3)}{2}}\\
			= & \szn  q^{\frac{n(n+1)}{2}} - 3q\szn \left(q^9 \right)^{\frac{n(n+1)}{2}}\\
			=&\psi(q) - 3q\psi(q^9).
			\end{align*} 
			Employing \eqref{eq1.12} in terms of eta functions in the above, we obtain the required result.
			\end{proof}

		\noindent \textbf{Acknowledgment:}\\
		
		The second author is supported by grant No.NSFDC/E-81088(Ref. No 231610177110)  by the funding agency NSFDC, INDIA, under UGC-NFSC-JRF.
		% The third author is supported by grant  No.09/0119(20939)/2025-EMR-I (Ref. No.Dec-23(i)/EU-V) by the funding agency CSIR, India, under CSIR-JRF. 
		The author is  grateful to his funding agency.\\

	\end{document}